\documentclass[10pt]{article}

\usepackage{pdfsync}
\usepackage{amsmath}\multlinegap=0pt
\usepackage[latin1]{inputenc}
\usepackage{amsthm}
\usepackage{amssymb}
\usepackage[francais,english]{babel}

\usepackage{amsmath}\multlinegap=0pt
\usepackage[latin1]{inputenc}
\usepackage{amsthm}
\usepackage{amssymb}
\usepackage[francais,english]{babel}

\numberwithin{equation}{section}
\theoremstyle{plain}
\newtheorem{thm}{Theorem}[section]
\newtheorem{coro}[thm]{Corollary}

\newtheorem{lem}[thm]{Lemma}

\theoremstyle{definition}

\theoremstyle{remark}
\newtheorem{rem}[thm]{Remark}

\newcommand\R{{\mathbb R}}
\newcommand\s{{\mathbb S}}

\newcommand{\ib}{\int_{B_R}}
\newcommand{\bs}{\begin{split}}
\newcommand{\es}{\end{split}}

\newcommand\vep{\varepsilon}
\newcommand\dd{\,\mbox{d} }

\def\F{{\cal F}}

\def\P{{\cal P}}

\def\div{\mbox{div} }

\title{A dual Moser-Onofri inequality and its extensions to higher dimensional spheres}

\author {Martial Agueh \thanks{\scriptsize  Department of Mathematics and Statistics
University of Victoria, Victoria, BC, V8W 3R4, CANADA,  \texttt{agueh@math.uvic.ca}.} \; 
Shirin Boroushaki  \thanks{\scriptsize  Department of Mathematics, 1984 Mathematics Road,
University of British Columbia, BC, V6T 1Z2, CANADA
  \texttt{shirinbr@math.ubc.ca}.}\; 
 and Nassif Ghoussoub \thanks{\scriptsize Department of Mathematics, 1984 Mathematics Road,
  University of British Columbia, BC, V6T 1Z2, CANADA
\texttt{nassif@math.ubc.ca}}
}
\date{\footnotesize Dedicated to Dominique Bakry on the occasion of his 60th birthday}

\begin{document}

\maketitle

\begin{abstract} We use optimal mass transport to provide a new proof and a dual formula to the Moser-Onofri inequality on $\s^2$ in the same spirit as the approach of Cordero-Erausquin, Nazaret and Villani \cite{CNV} to the Sobolev inequality and of  Agueh-Ghoussoub-Kang \cite{AGK} to more general settings. There are however many hurdles to overcome once a stereographic projection on $\R^2$ is performed:  Functions are not necessarily of compact support, hence boundary terms need to be evaluated. Moreover, the corresponding dual  free energy of the reference probability density $\mu_2(x)=\frac{1}{\pi(1+|x|^2)^2}$ is not finite on the whole space, which requires the introduction of a renormalized free energy into the dual formula.  We also extend this duality to higher dimensions and establish an extension of the Onofri inequality to spheres $\s^n$ with $n\geq 2$. What is remarkable is that the corresponding free energy is again given by $F(\rho)=-n\rho^{1-\frac{1}{n}}$, which means that both the {\it prescribed scalar curvature problem} and the {\it prescribed Gaussian curvature problem} lead essentially to the same dual problem whose extremals are stationary solutions of the fast diffusion equations.  

\end{abstract}



%%%%%%%%%%%%%%%%%%%%%%%%%%%%%%%%%%%%

\section{Introduction}
One of the equivalent forms of Moser's inequality \cite{moser} on the 2-dimensional sphere $\s^2$ states that the functional 
\begin{equation}\label{eqn 2sphere-onofri}
J(u):= \frac{1}{4}\int_{\s^2} |\nabla u|^2\dd\omega - \int_{\s^2} u\dd\omega- \log\left(\int_{\s^2} e^u\dd\omega\right) 
\end{equation}
is bounded below on $ H^1(\s^2)$,  
where 
$\dd\omega$ is the Lebesgue measure on $\s^2$, normalized so that $\int_{\s^2} \dd\omega=1$.  Later, Onofri \cite{onofri} showed that the infimum is actually zero, and that modulo conformal transformations, $u=0$ is the optimal function. Note that this inequality 
is related to the ``prescribed Gaussian curvature" problem on $\s^2$, 
\begin{equation}\label{eqn K-gaussian curvature}
\Delta u + K(x)e^{2u} =1 \quad \mbox{on} \;\; \s^2,
\end{equation}
where $K(x)$ is the Gaussian curvature associated to the metric $g=e^{2u}g_0$ on $\s^2$, and $\Delta = \Delta_{g_0}$ is the Laplace-Beltrami operator corresponding to the standard metric $g_0$. Finding $g$ for a given $K$ leads to solving (\ref{eqn K-gaussian curvature}). Variationally, this reduces to minimizing the functional 
\begin{equation}\label{eqn1}
\F(u)=\int_{\s^2} |\nabla u|^2\frac{\dd V_0}{4\pi} + 2\int_{\s^2} u\frac{\dd V_0}{4\pi} - \log\left(\int_{\s^2} K(x)e^{2u}\frac{\dd V_0}{4\pi}\right) \quad {\rm on}\;  H^1(\s^2),
\end{equation}
where the volume form is such that $\int_{\s^2} \dd V_0=4\pi$.  Onofri's result says that, modulo conformal transformations, $u\equiv 0$ is the only solution of the ``prescribed Gaussian curvature'' problem (\ref{eqn K-gaussian curvature}) for $K=1$, i.e., 
%\begin{equation}\label{eqn-0-gaussian curvature}
$\frac{1}{2}\Delta u + e^u = 1$ on $ \s^2, $
%\end{equation}
which after rescaling, $u\mapsto 2u$,  gives
\begin{equation}\label{eqn 2-gaussian curvature}
\Delta u + e^{2u} = 1 \quad \mbox{on} \;\; \s^2. 
\end{equation}
The proof given by Onofri in \cite{onofri} makes use of an ``exponential'' Sobolev inequality due to Aubin \cite{aubin} combined with the invariance of the functional  (\ref{eqn 2sphere-onofri}) under conformal transformations. Other proofs were given by Osgood-Philips-Sarnak \cite{OPS} and by Hong \cite{hong}. See also Ghoussoub-Moradifam \cite{GM}.

In this paper we shall give a proof based on mass transport. While this approach has by now become standard, there are many reasons why it hasn't been so far spelled out. The first originates in the fact that, unlike the case of $\R^n$,  optimal mass transport on the sphere is harder to work with. To avoid this difficulty, we use an equivalent formulation of the Onofri inequality (\ref{eqn 2sphere-onofri}), that we call  the Euclidean Onofri inequality on $\R^2$, namely
\begin{equation}\label{eqn 2euclidean-onofri}
\frac{1}{16\pi}\int_{\R^2} |\nabla u|^2\dd x  + \int_{\R^2} u\dd\mu_2 -\log\left(\int_{\R^2} e^u\dd\mu_2\right) \geq 0  \quad \forall u\in H^1(\R^2),
\end{equation}
which is obtained by projecting (\ref{eqn 2sphere-onofri}) on $\R^2$ via the stereographic projection with respect to the North pole $N=(0,0,1)$, i.e., $\Pi: \s^2 \rightarrow \R^2$, $\Pi(x):=\left(\frac{x_1}{1-x_3}, \frac{x_2}{1-x_3}\right)$ where $x=(x_1,x_2,x_3)$. Here $\mu_2$ is the probability density on $\R^2$ defined by $\mu_2(x)=\frac{1}{\pi(1+|x|^2)^2}$, and $\dd\mu_2 = \mu_2(x)\dd x$.

One can then try to apply the Cordero-Nazaret-Villani \cite{CNV} approach as generalized by Agueh-Ghoussoub-Kang \cite{AGK} and write the {\it Energy-Entropy production duality} for functions that are of compact support in $\Omega$, 
 \begin{eqnarray}
&&\sup\left\{- \int_{\Omega}(F(\rho) + \frac{1}{2}|x|^2\rho) \dd x; \int_\Omega \rho \dd x=1\right\}\\
&& \qquad \qquad=\inf\left\{\int_{\Omega} \alpha|\nabla u|^2
-G\left(\psi\circ u\right) \dd x; \, \int_\Omega\psi(u) \dd x=1\right\},\nonumber
\end{eqnarray}
where $G(x)=(1-n)F(x)+nxF'(x)$ and where $\psi$ and $\alpha$ are also computable from $F$.  

Recall that by choosing $F(x)=-nx^{1-1/n}$ and $\psi(t)=|t|^{2^*}$ where $2^*=\frac{2n}{n-2}$ and $n>2$, one obtains 
the following duality formula for the Sobolev inequality 
\begin{eqnarray}\label{eqn dual sob.}
\lefteqn{\sup\Big\{n\int_{\R^n}\rho^{1-1/n}\dd x - \frac{1}{2}\int_{\R^n} |x|^2\rho \dd x;  \;\int_{\R^n} \rho \dd x=1\Big\} }\nonumber\\
& & \qquad \qquad = \inf\Big\{2\left(\frac{n-1}{n-2}\right)^2 \int_{\R^n} |\nabla u|^2\dd x; \;\int_{\R^n} |u|^{2^*}\dd x =1\Big\},
\end{eqnarray}
where $u$ and $\rho$ have compact support in $\R^n$. The extremal $u_\infty$ and $\rho_\infty$ are then obtained as solutions of 
\begin{equation}\label{eqn140}
\nabla\left(\frac{|x|^{2}}{2} - \frac{n-1}{\rho_\infty^{1/n}}\right)=0, \quad \rho_\infty=u_\infty^{2^*}\in {\cal P}(\R^n).
\end{equation}
The best constants are then obtained 
by computing $\rho_\infty$ from (\ref{eqn140}) and inserting it into (\ref{eqn dual sob.}) in such a way that 
\[\inf\Big\{2\left(\frac{n-1}{n-2}\right)^2 \int_{\R^n} |\nabla u|^2\dd x: \;\; \int_{\R^n} |u|^{2^*}\dd x =1\Big\} = 
n\int_{\R^n}\rho_\infty^{1-1/n}\dd x - \frac{1}{2}\int_{\R^n} |x|^2\rho_\infty \dd x.\]
Note that this duality leads to a correspondence between a solution to the Yamabe equation 
\begin{equation}\label{eqn yamabe}
-\Delta u = |u|^{2^*-2}u \quad \mbox{on}\;\; \R^n
\end{equation}
and stationary solution to the rescaled fast diffusion  equation 
\begin{equation}\label{sob-fokker-planck}
\partial_t \rho = \Delta\rho^{1-\frac{1}{n}} + \div(x\rho) \quad \mbox{on}\;\; \R^n.
\end{equation}
The above scheme does not however apply to inequality (\ref{eqn 2euclidean-onofri}). For one, the functions $e^u\mu_2=\frac{e^{u(x)}}{\pi(1+|x|^2)^2}$ do not have compact support, and if one restricts them to bounded domains, we then need to take into consideration various boundary terms. What is remarkable is that a similar program can be carried out provided the dual formula involving the free energy  
$$J_\Omega (\rho)=- \int_{\Omega}(F(\rho) + |x|^2\rho) \dd x
$$ is renormalized by substituting it with 
$J_\Omega (\rho)-J_\Omega (\mu_2)$. 

Another remarkable fact is that the corresponding free energy turned out to be $F(\rho)=-2\rho^{\frac{1}{2}}$, which is the same as the one associated to the critical case of the Sobolev inequality $F(\rho)=-n\rho^{1-\frac{1}{n}}$ when $n\geq 3$. In other words, the Moser-Onofri inequality and the Sobolev inequality ``dualize" in the same way, and both the Yamabe problem (\ref{eqn yamabe}) and the prescribed Gaussian curvature problem (\ref{eqn 2-gaussian curvature}) reduce to the study of the fast diffusion equation (\ref{sob-fokker-planck}), 
with the caveat that in dimension $n=2$, the above equation needs to be considered only on bounded domains, with Neumann boundary conditions.

More precisely, we shall show that, when restricted to balls $B_R$ of radius $R$ in $\R^2$, there is a duality between the ``Onofri functional'' 
\[
\hbox{$I_R(u)=\frac{1}{16\pi} \int_{B_R} |\nabla u|^2\dd x + \int_{B_R} u\dd\mu_2$ \quad on \quad  $
X_R:=\{u\in H^1_0(B_R); \int_{B_R} e^u\dd\mu_2 =\ib \dd\mu_2\} $}
\]
and the free energy 
\[
\hbox{$J_R(\rho) = \frac{2}{\sqrt{\pi}}\int_{B_R} \sqrt{\rho}\dd x - \int_{B_R} |x|^2\rho\dd x$ \quad on \quad $Y_R:=\{\rho \in {\cal P}(\R^2);\, \ib \rho\dd x = \ib \dd \mu_2\}$}.
\]
Once the latter free energy is re-normalized, we then get
\begin{equation}
\sup\{J_R(\rho) - J_R(\mu_2); \rho \in Y_R\}=0=\inf\{ I_R(u); \, u\in X_R\}.
\end{equation}
Note that when $R\to +\infty$, the right hand side yields the Onofri inequality
\[
\inf\left\{\frac{1}{16\pi} \int_{\R^2} |\nabla u|^2\dd x + \int_{\R^2} u\dd\mu_2; \, u\in H^1_0(\R^2),  \int_{\R^2} e^u\dd\mu_2=1\right\} =0,
\]
while the left-hand side becomes undefined since $J_R(\mu_2) \to +\infty$ as $R\to \infty$.

We also show that our approach generalizes to  higher dimensions. More precisely,  if $B_R$ is a ball of radius $R$ in $\R^n$ where $n\geq 2$, and if one considers the probability density $\mu_n$ on $\R^n$ defined by
\[\mu_n(x)=\frac{n}{\omega_n(1+|x|^{\frac{n}{n-1}})^n}\]
($\omega_n$ is the area of the unit sphere in $\R^n$), and the operator $H_n(u,\mu_n)$ on $W^{1,n} (\R^n)$ by
\[H_n(u,\mu_n):=|\nabla u +\nabla(\log\mu_n)|^n - |\nabla(\log\mu_n))|^n - n |\nabla(\log\mu_n))|^{n-2}\nabla(\log\mu_n) \cdot \nabla u,\]
there is then a duality between the functional
\[I_R(u)=\frac{1}{\beta (n)}\int_{B_R} H_n(u,\mu_n)\dd x + \int_{B_R} u(x)\dd\mu_n \,\, \hbox{on $\, X_R:=\{u\in W^{1,n}_0(B_R); \int_{B_R} e^u\dd\mu_n =\ib \dd\mu_n$\}}\]
and the free energy -- renormalized by again subtracting $J_R(\mu_n)$ --
\[J_R(\rho) = \alpha (n) \int_{B_R} \rho(x)^{\frac{n-1}{n}}\dd x - \int_{B_R} |x|^{\frac{n}{n-1}}\rho(x)\dd x  \hbox{\,  on \, $Y_N:=\{\rho \in {\cal P}(\R^n);\, \ib \rho\dd x = \ib \dd \mu_n\}$}.\]
where $\alpha (n)=\frac{n}{n-1}(\frac{n}{\omega_n})^{1/n}$ and $\beta (n)=\omega_n\,(\frac{n}{n-1})^{n-1}n^n$. We then deduce the following higher dimensional version of the Euclidean Onofri inequality in $\R^n$, $n\geq 2$:
\begin{equation}
 \frac{1}{\beta (n)} \int_{\R ^n} H_p (u,\mu_n) \dd x +\int_{\R ^n} u \dd\mu_n -\log \left(\int_{\R ^n}  e^u \dd\mu_n \right)\geq 0\quad  \hbox{for all $u\in W^{1,n}(\R^n)$,}
 \end{equation}
and $u\equiv 0$ is the extremal.

\smallskip

We finish this introduction by mentioning that the Euclidean Onofri inequality (\ref{eqn 2euclidean-onofri}) in the radial case was recently obtained in \cite{DEJ} using mass transport techniques. However,  to our knowledge, our duality result, the extensions of Onofri's inequality to higher dimensions, as well as the mass transport proof of the general (non-radial) Onofri inequality are new.

\smallskip 

The paper is organized as follows. In section \ref{sec1}, we recall the mass transport approach to sharp Sobolev inequalities and some consequences. In section \ref{sec2}, we establish the mass transport duality principle leading to the Euclidean Onofri inequality (\ref{eqn 2euclidean-onofri}). Finally, in section \ref{sec3}, we extend the 2-dimensional analysis of section \ref{sec2} to higher dimensions $n\geq 2$, and we produce a n-dimensional version of the Euclidean Onofri inequality, as well as other interesting inequalities.

%%%%%%%%%%%%%%%%%%%%%%%%%%%%%%%%%%%%

\section{Preliminaries} \label{sec1}
We start by briefly describing the mass transport approach to sharp Sobolev inequalities as proposed by \cite{CNV}. We will follow here the framework of \cite{AGK} as it clearly shows the correspondence between the Yamabe equation (\ref{eqn yamabe}) and the rescaled fast diffision equation (\ref{sob-fokker-planck}).

 Let $\rho_0, \rho_1\in {\cal P}(\R^n)$. If $T$ is the optimal map pushing $\rho_0$ forward to $\rho_1$ (i.e. $T_{\#}\rho_0=\rho_1$) in the mass transport problem for the quadratic cost $c(x-y)=\frac{|x-y|^2}{2}$ (see \cite{vi} for details), then $[0,1]\ni t\mapsto \rho_t=(T_t)_{\#}\rho_0$ is the ``geodesic'' joining $\rho_0$ and $\rho_1$ in $\left({\cal P}(\R^n), d_2\right)$; here $T_t:=(1-t)\mbox{id}+tT$ and $d_2$ denotes the quadratic Wasserstein distance (see \cite{vi}). Moreover, given a function $F:[0,\infty)\rightarrow \R$ such that $F(0)=0$ and $x\mapsto x^nF(x^{-n})$ is convex and non-increasing, the functional $H^F(\rho):=\int_{\R^n} F\left(\rho(x)\right)\dd x$ is displacement convex \cite{mc}, in the sense that $[0,1]\ni t\mapsto H^F(\rho_t)\in \R$ is convex (in the usual sense), for all pairs $(\rho_0,\rho_1)$ in ${\cal P}(\R^n)$. A direct consequence is the following convexity inequality, known as ``{\em energy inequality}'':
\[H^F(\rho_1)-H^F(\rho_0) \geq \left[\frac{d}{dt}H^F(\rho_t)\right]_{t=0} = \int_{\R^n} \rho_0\nabla\left(F'(\rho_0)\right)\cdot\left(T-\mbox{id}\right)\dd x,\]
which, after integration by parts of the right hand side term, reads as
\begin{equation}\label{eqn7}
-H^F(\rho_1) \leq - H^{F+nP_F}(\rho_0)  -\int_{\R^n}\rho_0\nabla\left(F'(\rho_0)\right)\cdot T(x)\,\mbox{d}x,
\end{equation}
where $P_F(x)=xF'(x)-F(x)$; here $\mbox{id}$ denotes the identity function on $\R^n$. By the Young inequality
\begin{equation}\label{eqn9}
-\nabla\left(F'(\rho_0)\right)\cdot T(x)\leq  \frac{|\nabla F'(\rho_0)|^p}{p} + \frac{|T(x)|^{q}}{q}  \quad \forall p, q >1 \;\; \mbox{such that} \quad \frac{1}{p}+\frac{1}{q}=1,
\end{equation}
(\ref{eqn7}) gives
\[-H^F(\rho_1)\leq -H^{F+nP_F}(\rho_0)+\frac{1}{p}\int_{\R^n}\rho_0 |\nabla F'(\rho_0)|^p \,\mbox{d}x + \frac{1}{q}\int_{\R^n}\rho_0(x) |T(x)|^{q}\,\mbox{d}x,\]
i.e.,
 \begin{equation}\label{eqn10}
 -H^F(\rho_1) - \frac{1}{q}\int_{\R^n} |y|^q\rho_1(y)\dd y  \leq -H^{F+nP_F}(\rho_0) + \frac{1}{p}\int_{\R^n}\rho_0 |\nabla F'(\rho_0)|^{p} \,\mbox{d}x,
 \end{equation}
 where we use that $T_{\#}\rho_0=\rho_1$.  Furthermore, if $\rho_0=\rho_1$, then $T=\mbox{id}$ and equality holds in (\ref{eqn7}). Then equality holds in (\ref{eqn10}) if it holds in the Young inequality (\ref{eqn9}). This occurs when $\rho_0=\rho_1$ satisfies 
 $\nabla\left(F'\left(\rho_0(x)\right)+\frac{|x|^{q}}{q}\right)=0$.
 Therefore, we have established the following duality:
 \begin{eqnarray}\label{eqn11}
\lefteqn{\sup\Big\{-H^F (\rho_1) - \frac{1}{q}\int_{\R^n} |y|^q\rho_1(y)\dd y: \;\;\rho_1\in{\cal P}(\R^n)\Big\} }\nonumber\\
& & \qquad \qquad  = \inf\Big\{ -H^{F+nP_F}(\rho_0) + \frac{1}{p} \int_{\R^n}\rho_0 |\nabla F'(\rho_0)|^p \,\mbox{d}x: \; \rho_0\in{\cal P}(\R^n)\Big\},
 \end{eqnarray}
and an optimal function in both problems is $\rho_0=\rho_1:=\rho_\infty$ solution of 
\begin{equation}\label{eqn12}
\nabla\left(F'\left(\rho_\infty(x)\right)+\frac{|x|^{q}}{q}\right)=0.
\end{equation}
In particular, choosing $F(x)=-nx^{1-1/n}$ and $\rho_0=u^{2^*}$ where $2^*=\frac{2n}{n-2}$ and $n>2$, then $H^{F+nP_F}=0$, and (\ref{eqn11})-(\ref{eqn12}) gives the duality formula for the Sobolev inequality (\ref{eqn dual sob.}).

\smallskip

Our goal now is to extend this mass transport proof of the Sobolev inequality to the Euclidean Onofri inequality (\ref{eqn 2euclidean-onofri}). As already mentioned in the introduction, a first attempt on this issue was recently made by \cite{DEJ}, but the result produced was only restricted to the radial case. Here we show in full generality (without restricting to radial functions $u$) that the Euclidean Onofri inequality (\ref{eqn 2euclidean-onofri}) can be proved by mass transport techniques. More precisely, we establish an analogue of the duality (\ref{eqn dual sob.}) for Euclidean Onofri inequality (see Theorem  \ref{theo 2-duality}), from which we deduce the Onofri inequality (\ref{eqn 2euclidean-onofri}) and its optimal function (see Theorem  \ref{theo 2-euclidean-onofri}). Furthermore, we obtain as for the Sobolev inequality, a correspondence between the prescribed Gaussian curvature problem (\ref{eqn 2-gaussian curvature}) and the rescaled fast diffusion equation (\ref{sob-fokker-planck}). Finally, we extend our analysis to higher dimensions, and then produce a new version of the Onofri inequality in dimensions $n\geq 2$ (see Theorem \ref{theo n-euclidean-onofri}).

\bigskip

The following general lemma from mass transport theory will be needed in the sequel.

\begin{lem}\label{lemma1}
 Let $\rho_0, \rho_1 \in \P(B_R)$, where $\P(B_R)$ denotes the set of probability densities on the ball $B_R\subset \R^n$. Let $T$ be the optimal map pushing $\rho_0$ forward to $\rho_1$ (i.e. $T_{\#}\rho_0=\rho_1)$  in the mass transport problem for the quadratic cost. Then 
\begin{equation}\label{eqn 2.1}
 \int_{B_R} \rho_1(y) ^{1-\frac{1}{n}}  \dd y \leq \frac{1}{n} \int_{B_R} \rho_0(x) ^ {1-\frac{1}{n}} {\rm div}\, (T(x)) \dd x .
\end{equation}
 \end{lem} 

\begin{proof} By \cite{brenier}, the optimal map $T: B_R\rightarrow B_R$ is such that $T=\nabla\varphi$ where $\varphi: B_R\rightarrow \R$ is convex.  Since  $T_{\#}\rho_0=\rho_1$, then it satisfies the Monge-Amp\`ere equation \begin{equation}\label{monge0}
  \rho_0(x)=\rho_1(T(x)) \det {\nabla T(x)}
  \end{equation}  
 or equivalently
 \begin{equation}\label{monge1}
 \rho_1 (T(x))=\rho_0 (x) [\det {\nabla T(x)}]^{-1}.
 \end{equation}
 By the arithmetic-geometric-mean inequality
\[  [\det \nabla T(x)]^ {\frac{1}{n}}\leq \frac{1}{n} \div (T(x)),\]
 (\ref{monge1}) gives
\begin{equation}\label{monge3}
\rho_1 (T(x))^{-\frac{1}{n}}   \leq \frac{1}{n} \, \rho_0(x)^{-\frac{1}{n}} \div\left(T(x)\right).
 \end{equation}
 Now using the change of variable $y=T(x)$, we have
\[\int_{B_R} \rho_1(y) ^{1-\frac{1}{n}} \dd y  = \int_{B_R} \rho_1\left(T(x)\right) ^{1-\frac{1}{n}} \det (\nabla T(x)) \dd x, \]
which implies by (\ref{monge0}) and (\ref{monge3}), 
\[\int_{B_R} \rho_1(y) ^{1-\frac{1}{n}}  \dd y  \leq \frac{1}{n} \int_{B_R} \rho_0(x)^{-\frac{1}{n}} \div (T(x)) \rho_0 (x) \dd x = \frac{1}{n} \int_{B_R} \rho_0 (x) ^{1-\frac{1}{n}} \div (T(x)) \dd x.\]
\end{proof} 

%%%%%%%%%%%%%%%%%%%%%%%%%%%%%%%%%%%%%%%%%%%%

\section{Euclidean 2-dimensional Onofri inequality: a duality formula}\label{sec2}

Consider the probability density \[\mu_2(y)=\frac{1}{\pi\left(1+|y|^2\right)^2} \quad \forall y\in\R^2,\]
 and set
\begin{equation}\label{integral theta}
\theta_R:=\int_{B_R} \mu_2(y)\dd y = \frac{R^2}{1+R^2},
\end{equation}
where $B_R\subset\R^2$. We now establish the following duality.

\begin{thm}{\bf (Duality for the 2-dimensional Euclidean Onofri inequality)} \label{theo 2-duality}

Consider the functionals 
\[I_R(u) = \frac{1}{16\pi} \int_{B_R} |\nabla u(x)|^2\dd x + \int_{B_R} u(x)\dd \mu_2(x)\hbox{\quad  {\rm on}\,  $H^1_0(B_R)$, }\]
and
\[J_R(\rho)=\frac{2}{\sqrt{\pi}} \int_{B_R} \sqrt{\rho(y)}\dd y - \int_{B_R} |y|^2\rho(y)\dd y\hbox{ \quad {\rm on} ${\cal P}(\R^2)$.}\]
Then, the following holds:
\begin{eqnarray}\label{eqn duality-2-onofri}
&&\sup\left\{ J_R(\rho) - J_R(\mu_2);\,  \rho\in {\cal P}(\R^2),\,  \int_{B_R}\rho\dd y
=\theta_R \right\} \nonumber\\
&&\qquad \qquad \qquad= \inf\left\{ I_R(u);\, u\in H^1_0(B_R),\,  \int_{B_R}e^u\dd\mu_2 =\theta_R\right\}.
\end{eqnarray}
Moreover, the maximum on the l.h.s. is attained at $\rho_{max}=\mu_2$, and the minimum on the r.h.s. is attained at $u_{min} = 0$. Furthermore, the maximum and minimum value of  both functionals in (\ref{eqn duality-2-onofri}) is $0$. 
\end{thm}

\begin{proof} Applying lemma \ref{lemma1} in dimension $n=2$, and integrating by parts, we get
\begin{align}\label{5}
\begin{split}
2 \int_{B_R} \sqrt{\rho_1 (y)} \ dy & \leq \int_{B_R} \sqrt{\rho_0 (x)}\ \div (T(x)) \dd x \\
&= - \int_{B_R} \nabla (\sqrt{\rho_0}) \cdot T(x) \dd x + \int_{\partial B_R} \sqrt{\rho_0 (x)} \ T(x) \cdot \nu \dd S 
\end{split}
\end{align}
where $\nu=x/|x|$ denotes the outward unit normal to the boundary $\partial B_R$ of $B_R$, and $\dd S$ is the surface element.  By the basic identity
\[\nabla\left(\sqrt{\rho_0}\right) = \frac{1}{2} \sqrt{\rho_0} \ \nabla (\log \rho_0),\]
we can rewrite (\ref{5}) as
 \begin{equation}\label{7}
 4 \int_{B_R} \sqrt{\rho_1 (y)} \dd y \leq - \int_{B_R} \sqrt{\rho_0} \ \nabla (\log \rho_0)\cdot T(x) \dd x + 2 \int_{\partial B_R} \sqrt{\rho_0 (x)} \ T(x) \cdot \nu \dd S.
 \end{equation}
Now introduce the probability density 
\[\rho_0 (x)= \frac{e^u(x) \mu_2(x)}{\theta_R} \quad \forall x\in B_R,\]
where $u\in H^1_0(B_R)$ satisfies $\int_{B_R}e^u\dd\mu_2 =\theta_R$. Then (\ref{7}) becomes
 \begin{equation}\label{8}
  4 \int_{B_R} \sqrt{\theta_R\, \rho_1} \dd y \leq - \int_{B_R} \sqrt{e^u \mu_2} \ \nabla \left(\log (e^u \mu_2)\right)\cdot T(x) \dd x + 2 \int_{\partial B_R} \sqrt{e^u \mu_2} \ T(x) \cdot \nu \dd S.
 \end{equation}
We use Young inequality
\[-  \sqrt{e^u \mu_2} \ \nabla \left(\log (e^u \mu_2)\right)\cdot T(x) \leq \frac{1}{2\varepsilon} |\nabla \left(\log (e^u \mu_2)\right)|^2 + \frac{\varepsilon}{2} \ e^u \mu_2 |T(x)|^2 \quad \mbox{with} \quad \vep=4\sqrt{\pi},\]
and the fact that $T_{\#}\rho_0=\rho_1$, to rewrite (\ref{8}) as
\begin{eqnarray}\label{9}
32\int_{B_R} \sqrt{\pi\theta_R\, \rho_1} \ dy  - 16\pi\int_{B_R}|y|^2\theta_R\rho_1(y) \dd y - 16\int_{\partial B_R} \sqrt{\pi e^u \mu_2} \ T(x) \cdot \nu \dd S  \nonumber \\
 \leq  \int_{B_R} |\nabla u+\nabla(\log\mu_2)|^2\dd x. \qquad \qquad \qquad \qquad
\end{eqnarray}
It is easy to check that 
\[\ib |\nabla u+ \nabla (\log \mu_2))|^2 \dd x = \ib |\nabla u|^2 \dd x +\ib |\nabla (\log \mu_2)|^2 \dd x - 2 \ib  u\, \Delta (\log \mu_2) \dd x,\]
and
\[\Delta (\log \mu_2) = -8\pi\mu_2.\]
Inserting these identities into (\ref{9}) and using that $u\vert_{\partial B_R} =0$,  we have
\begin{eqnarray}\label{10}
32\int_{B_R} \sqrt{\pi\theta_R\, \rho_1} \ dy  - 16\pi\int_{B_R}|y|^2\theta_R\rho_1(y) \dd y -  16\int_{\partial B_R} \sqrt{\pi \mu_2} \ T(x) \cdot \nu \dd S\nonumber \\
 \leq  \ib |\nabla u|^2 \dd x  +16\pi \ib  u \dd\mu_2 +     \ib |\nabla (\log \mu_2)|^2 \dd x.  \qquad \qquad
\end{eqnarray}
We now estimate the boundary term. Since $T: B_R \rightarrow B_R$, then $|T(x)|\leq R$ for all $x\in B_R$, so that
 \[ \ \int_{\partial B_R} T(x)\cdot x \dd S \leq \int_{\partial B_R} |T(x)||x| \dd S \leq R^2 \int_{\partial B_R} \dd S =2 \pi R^3.\]
 Then using that $\nu=x/|x|$, we have
 \begin{equation}\label{11}
 -\int_{\partial B_R} \sqrt{\pi \mu_2} \ T(x) \cdot \nu \dd S =  -\frac{1}{R(1+R^2)}\int_{\partial B_R} T(x)\cdot x\dd S \geq -\frac{2\pi R^2}{1+R^2} = -2\pi\theta_R.
 \end{equation}
We insert (\ref{11}) into (\ref{10}) to have
\begin{eqnarray}\label{12}
32\int_{B_R} \sqrt{\pi\theta_R\, \rho_1} \ dy  - 16\pi\int_{B_R}|y|^2\theta_R\rho_1(y) \dd y -  32\pi\theta_R \nonumber \\
 \leq  \ib |\nabla u|^2 \dd x  +16\pi \ib  u \dd\mu_2 +     \ib |\nabla (\log \mu_2)|^2 \dd x.
\end{eqnarray}
By a direct computation, we have
\[\ib |\nabla(\log\mu_2)|^2 \dd x = 16 \ib \sqrt{\pi\mu_2} \dd y  - 16\pi\theta_R.\]
This implies that
\begin{eqnarray*}
32\int_{B_R} \sqrt{\pi\theta_R\, \rho_1} \ dy  - 16\pi\int_{B_R}|y|^2\theta_R\rho_1(y) \dd y - 16\ib \sqrt{\pi\mu_2}\dd y -16\pi\theta_R \nonumber \\
 \leq  \ib |\nabla u|^2 \dd x  +16\pi \ib  u \dd\mu_2, \qquad \qquad \qquad \qquad
\end{eqnarray*}
that is, after dividing by 16,
\begin{eqnarray}\label{13}
2\int_{B_R} \sqrt{\pi\theta_R\, \rho_1} \ dy  - \pi\int_{B_R}|y|^2\theta_R\rho_1 \dd y - \left(\ib \sqrt{\pi\mu_2}\dd y +\pi\theta_R\right) \nonumber \\
 \leq  \frac{1}{16}\ib |\nabla u|^2 \dd x  +\pi \ib  u \dd\mu_2. \qquad \qquad \qquad \qquad
\end{eqnarray}
We easily check that 
\[\ib \sqrt{\pi\mu_2}\dd y + \pi\theta_R = 2\ib \sqrt{\pi\mu_2}\dd y - \pi \ib |y|^2\mu_2 \dd y.\]
Inserting this identity into (\ref{13}) yields
\begin{eqnarray}\label{14}
2\int_{B_R} \sqrt{\pi\theta_R\, \rho_1} \ dy  - \int_{B_R}|y|^2\pi\theta_R\rho_1 \dd y - \left(2\ib \sqrt{\pi\mu_2}\dd y - \ib |y|^2\pi\mu_2 \dd y\right) \nonumber \\
 \leq  \frac{1}{16}\ib |\nabla u|^2 \dd x  +\pi \ib  u \dd\mu_2. \qquad \qquad \qquad \qquad
\end{eqnarray}
Now, we simplify (\ref{14}) by $\pi$ and set $\rho=\theta_R\rho_1$, to have that
\begin{equation}\label{15}
J_R(\rho) - J_R(\mu_2) \leq I_R(u) \
\end{equation}
for all functions $u$ and $\rho$ such that $u\in H^1_0(B_R)$ and $\ib e^u\dd\mu_2 = \theta_R=\ib \rho(y) \dd y$. Observe that if $\rho=\mu_2$, then the l.h.s of (\ref{15}) is $0$, and so is the r.h.s if $u=0$. Therefore
\[0\leq \sup_{\rho} \left\{J_R(\rho) - J_R(\mu_2): \; \ib \rho(y) \dd y =\theta_R\right\} \leq \inf_{u\in H^1_0(B_R)}\left\{ I_R(u): \ib e^u \dd \mu_2 =\theta_R\; \right \} \leq 0.\]
This completes the proof of the theorem.
\end{proof}

\begin{rem}
Note that the duality formula (\ref{eqn duality-2-onofri}) for Onofri's inequality is restricted to balls $B_R$, while for the Sobolev inequality, it is defined on the whole space $\R^n$, $n\geq 3$. The reason is simply because the free energy
\[J_R(\rho)= \frac{2}{\sqrt{\pi}} \int_{B_R} \sqrt{\rho(y)}\dd y - \int_{B_R} |y|^2\rho(y)\dd y\]
for (\ref{eqn duality-2-onofri}) is not always finite on the whole domain $\R^2$, that is when $R\to \infty$. This can be seen when we choose for example $\rho=\mu_2$. We then have 
\[J_R(\mu_2) = \log(1+R^2) + \frac{R^2}{1+R^2}\]
so that $\lim_{R \to \infty} J_R(\mu_2) = \infty$.
\end{rem}

\smallskip

From Theorem  \ref{theo 2-duality}, we can deduce a certain number of results. One of them is the Euclidean Onofri inequality.

\begin{thm} ({\bf 2-dimensional Euclidean Onofri inequality})\label{theo 2-euclidean-onofri}
 
The following holds:
\begin{equation}\label{L2 onofri}
 \log\left(\int_{\R^2} e^u\,\mbox{d}\mu_2\right) -  \int_{\R^2}u\, d\mu_2   \leq \frac{1}{16\pi}\int_{\R^2} |\nabla u|^2 dx  \quad \forall u\in H^1(\R^2).
\end{equation}
Moreover $u=0$ is an optimal function.
\end{thm}

\begin{proof} By Theorem  \ref{theo 2-duality}, we have for all $u\in H^1_0(B_R)$ such that $\ib e^u\dd\mu_2 =\theta_R=\frac{R^2}{1+R^2}$,
\begin{equation}\label{16}
I_R(u) =  \frac{1}{16\pi} \int_{B_R} |\nabla u(x)|^2\dd x + \int_{B_R} u(x)\dd \mu_2(x) \geq 0,
\end{equation}
with equality if $u=0$. Letting $R\to\infty$, we have $\theta_R\to 1$; then (\ref{16}) becomes
\[\frac{1}{16\pi} \int_{\R^2} |\nabla u(x)|^2\dd x + \int_{\R^2} u(x)\dd \mu_2(x) \geq 0 = \log\left(\int_{\R^2} e^u\dd\mu_2\right),\]
for all $u\in H^1(\R^2)$, with equality if $u=0$. This completes the proof of the theorem.
\end{proof}

Another consequence of Theorem  \ref{theo 2-duality} is the correspondence between solutions to the prescribed Gaussian curvature problem and stationary solutions to a  rescaled fast diffusion equation.
Indeed, it is known from the theory of mass transport (see \cite{vi} for example) that a maximizer $\rho$ to the supremum problem in (\ref{eqn duality-2-onofri}) is a solution to the pde:
\[\frac{1}{2\sqrt{\pi}} \Delta(\sqrt{\rho}) + \div(x\rho)=0 \quad \mbox{in} \quad  B_R,\]
(with Neumann boundary condition) that is the stationary solution of the rescaled fast diffusion equation
\begin{equation}\label{rescaled fd}
\partial_t\rho = \frac{1}{2\sqrt{\pi}} \Delta(\sqrt{\rho}) + \div(x\rho) \quad \mbox{in}\quad B_R.
\end{equation}
On the other hand, it can be shown that a minimizer $u$ to the infimum problem in (\ref{eqn duality-2-onofri}) solves the semilinear pde:
\begin{equation}\label{gaussian curvature projected}
\frac{1}{8\pi}\Delta u +\lambda \mu_2\, e^u = \mu_2 \quad \mbox{in}\quad B_R,
\end{equation}
where $\lambda$ is the Lagrange multiplier for the constraint $\int_{B_R} e^u\dd \mu_2=\theta_R$. The duality formula (\ref{eqn duality-2-onofri}) establishes then a correspondence between stationary solutions to the rescaled fast diffusion equation (\ref{rescaled fd}) and solutions to the semilinear equation (\ref{gaussian curvature projected}).

%%%%%%%%%%%%%%%%%%%%%%%%%%%%%%%%%%%%%%%%%%%%

\section{Extension to higher dimensions} \label{sec3}

Here, we extend the 2-dimensional analysis of section \ref{sec2} to higher dimensions $n\geq 2$, thereby producing a $n$-dimensional version of the Euclidean Onofri inequality (\ref{L2 onofri}). We start by proving a $n$-dimensional version of the duality formula (\ref{eqn duality-2-onofri}).  Hereafter, we denote by 
\[\mu_n(y)=\frac{n}{\omega_n(1+|y|^\frac{n}{n-1})^n}\]
 a probability density on $\R^n$, where $\omega_n$ is the area of the unit sphere in $\R^n$.
 Set 
\[\theta_R :=\ib \mu_n(y)\dd y = \frac{R^n}{(1+R^\frac{n}{n-1})^{n-1}}.\]

\begin{thm}{\bf (Duality for n-dimensional Euclidean Onofri inequality)} \label{theo n-duality}

Consider the functionals
\[J_R(\rho) = \alpha(n) \int_{B_R} \rho(y)^{\frac{n-1}{n}}\dd y - \int_{B_R} |y|^\frac{n}{n-1}\rho(y)\dd y \]
and
\[I_R(u)= \frac{1}{\beta(n)}\int_{B_R} H_n(u,\mu_n)\dd x + \int_{B_R} u(x)\dd\mu_n \qquad \forall u\in W^{1,n}_0(B_R)\]
where 
\[ \alpha(n)= \frac{n}{n-1}\left(\frac{n}{\omega_n}\right)^{1/n}, \qquad \beta(n)=\left(\frac{n}{n-1}\right)^{n-1}\, n^n\, \omega_n\]
and
\[H_n(u,\mu_n):=|\nabla u +\nabla(\log\mu_n)|^n - |\nabla(\log\mu_n)|^n - n |\nabla(\log\mu_n)|^{n-2}\nabla(\log\mu_n) \cdot \nabla u.\]
Then the following duality holds:
\begin{equation}\label{eqn duality-n-onofri}
\sup\left\{ J_R(\rho) - J_R(\mu_n); \rho \in {\cal P}(\R^n), \, \int_{B_R}\rho\dd y =\theta_R \right\} = \inf\left\{ I_R(u);\, u\in W^{1,n}_0(B_R),   \int_{B_R}e^u\dd\mu_n =\theta_R\right\}.
\end{equation}
Moreover, the maximum on the l.h.s. is attained at $\rho_{max}=\mu_n$, and the minimum on the r.h.s. is attained at $u_{min} = 0$. Furthermore, the maximum and minimum value of  both functionals in (\ref{eqn duality-n-onofri}) is $0$. 
\end{thm}

\begin{rem} Before proving the theorem, we make a few  remarks on the operator $H_n(u,\mu_n)$.
\begin{enumerate}
\item Consider the function $c:\R^n \rightarrow \R$, $c(z)=|z|^n$, $n\geq 2$. Clearly $c$ is strictly convex, and $\nabla c(z)=n|z|^{n-2}z$. So we have the convexity inequality
\begin{equation}\label{inequality c}
c(z_1) - c(z_0) -\nabla c(z_0) \cdot (z_1-z_0) \geq 0 \quad \forall z_0, z_1\in \R^n.
\end{equation}
Setting $z_0=\nabla(\log\mu_n)$ and $z_1=\nabla u +\nabla(\log\mu_n)$, we see that $H_n(u,\mu_n)$ is nothing but the l.h.s of (\ref{inequality c}); we then deduce that
\[ H_n(u,\mu_n) \geq 0 \qquad  \forall \, u, \mu_n.\]
\item For all $u\in  W^{1,n}_0(B_R)$, the integral of $H_n(u,\mu_n)$ over $B_R$ involves a well-known operator, the $n$-Laplacian $\Delta_n$, defined by
\begin{equation}\label{eqn n-laplacian}
\Delta_n v := \div(|\nabla v|^{n-2} \nabla v).
\end{equation}
Indeed, this can be seen after performing an integration by parts in the last term of $H_n(u,\mu_n)$,
\[\ib H_n(u,\mu_n) \dd y =  \ib |\nabla u +\nabla \log \mu_n|^n \dd y - \ib |\nabla(\log\mu_n))|^n\dd y   + n\ib u\, \Delta_n(\log\mu_n)\dd y.\]
\end{enumerate}
\end{rem}

\begin{proof} (Theorem \ref{theo n-duality}). It goes along the lines of the proof of Theorem  \ref{theo 2-duality}. Indeed, applying lemma \ref{lemma1}  in dimension $n\geq 2$, and using an integration by parts and $\frac{1}{m}:=1-\frac{1}{n}$, we have
\[n\ib \rho_1 ^{1/m} \dd y \leq - \ib \nabla(\rho_0 ^{1/m}) \cdot T(x) \dd x +  \int_{\partial B_R} \rho_0 ^{1/m} T(x)\cdot\nu \dd S.\]
We use the elementary identity  
\[\nabla (\rho_0 ^{1/m})= \frac{1}{m} \rho_0 ^ {1/m} \nabla (\log \rho_0)\]
to obtain
\begin{equation}\label{first}
mn \ib \rho_1 ^{1/m}\dd y \leq - \ib \rho_0 ^{1/m}  \nabla (\log \rho_0) \cdot T(x) \dd x + m \int_{\partial B} \rho_0 ^{1/m} T(x)\cdot \nu \dd S.
\end{equation}
Set $\rho_0 = \frac{e^{u} \mu_n}{\theta_R}$, where $u\in W^{1,n}_0(B_R)$ satisfies $\ib e^{u} d\mu_n =\theta_R$. 
Then (\ref{first}) becomes
\[mn \ib (\theta_R \rho_1) ^{1/m}\dd y \leq - \ib (e^u \mu_n) ^{1/m}  \nabla (\log (e^u \mu_n)) \cdot T(x) \dd x + m \int_{\partial B_R} (e^u \mu_n) ^{1/m} T(x)\cdot \nu \dd S.\]
Using Young's inequality 
\[-\left(e^u \mu_n\right)^{1/m} \nabla \left(\log (e^u \mu_n)\right)\cdot T(x) \leq \frac{1}{n\varepsilon} |\nabla \left(\log (e^u \mu_n)\right)|^n + \frac{\varepsilon^{m/n}}{m} \ e^u \mu_n |T(x)|^m \quad \forall \vep>0\]
and the fact that $T_{\#}\rho_0=\rho_1$, we get 
\begin{eqnarray}\label{second}
mn^2\varepsilon\int_{B_R} \left(\theta_R\rho_1\right)^{1/m} \dd y -mn\varepsilon\int_{\partial B_R} \mu_n^{1/m} T(x)\cdot \nu \dd S -\frac{n}{m}\varepsilon^m\int_{B_R}|y|^m \theta_R\rho_1 dy\nonumber \\
 \leq \int_{B_R} |\nabla u +\nabla\log\mu_n|^n \dd x. \qquad \qquad \qquad \qquad
\end{eqnarray}
We can estimate the boundary term as in the 2-dimensional case. We have

\begin{align}\label{boundary1}
\bs
\int_{\partial B_R} \mu_n^{1/m} T(x)\cdot \nu \dd S &= \left(\frac{n}{\omega_n}  \right)^ {1/m} \frac{1}{(1+R^m)^{n/m}} \int_{\partial B_R} T(x)\cdot \frac{x}{|x|} \dd S \\
& \leq n^{1/m} \omega_n^ {1/n} \frac{R^n}{(1+R^m)^{n/m}} = n^{1/m} \omega_n^ {1/n}\,\theta_R.
\es
\end{align}
Inserting (\ref{boundary1}) into (\ref{second}), and setting $\rho :=\theta_R \rho_1$, we get for all $ \vep>0$, 
\begin{eqnarray}\label{2nd}
\varepsilon \left[n^2 m \ib \rho^{1/m} \dd y-  n^{1+1/m}m \omega_n^{1/n} \theta_R \right] - \varepsilon^m \frac{n}{m} \ib |y|^m \rho \dd y\ \nonumber \\
\leq \ib |\nabla u +\nabla \log \mu_n|^n \dd x. 
\end{eqnarray}
Now, we introduce the operator $H_n(u,\mu_n)$ in the r.h.s of (\ref{2nd}). We have
\[|\nabla u +\nabla \log \mu_n|^n = H_n(u,\mu_n) +  |\nabla(\log\mu_n))|^n + n |\nabla(\log\mu_n))|^{n-2}\nabla(\log\mu_n) \cdot \nabla u,\]
which, after an integration by parts, yields
\[\ib |\nabla u +\nabla \log \mu_n|^n \dd x = \ib H_n(u,\mu_n) \dd x +\ib |\nabla(\log\mu_n))|^n\dd x   - n\ib u\, \Delta_n(\log\mu_n)\dd x,\]
where $\Delta_n$ is the $n$-Laplacian operator defined by (\ref{eqn n-laplacian}).
 By a direct computation, 
\[\Delta_n(\log\mu_n) = -(nm)^{n-1}\omega_n\mu_n.\]
Then
\[\ib |\nabla u +\nabla \log \mu_n|^n \dd x = \ib H_n(u,\mu_n) \dd x + n^nm^{n-1}\omega_n\ib u\dd\mu_n +\ib |\nabla(\log\mu_n))|^n\dd x,\]
and so (\ref{2nd}) becomes for all $\vep>0$, 
\begin{eqnarray}\label{2nd-bis}
\lefteqn{ \varepsilon \left[n^2 m \ib \rho^{1/m} \dd y-  n^{1+1/m}m \omega_n ^{1/n} \theta_R \right] - \varepsilon^m \frac{n}{m} \ib |y|^m \rho \dd y} \nonumber \\
& & \leq \ib H_n(u,\mu_n) \dd x + n^nm^{n-1}\omega_n\ib u\dd\mu_n +\ib |\nabla(\log\mu_n)|^n\dd x.  
\end{eqnarray}
Next, we focus on the l.h.s of (\ref{2nd-bis}). For convenience, we denote 
\[A_\rho := n^2m \int_{B_R} \rho^{1/m} \dd y -n^{(1+1/m)}m\omega_n^{1/n}\theta_R, \quad B_\rho:= \frac{n}{m}\int_{B_R}|y|^m\rho \dd y,\]
and
\[G_\rho(\vep) := \vep A_\rho - \epsilon^m B_\rho.\]
Then (\ref{2nd-bis}) reads as
\begin{equation}\label{2nd-prime}
G_\rho(\vep) \leq \ib H_n(u,\mu_n) \dd y + n^nm^{n-1}\omega_n\ib u\dd\mu_n +\ib |\nabla(\log\mu_n))|^n\dd y  \quad  \forall \vep>0.
\end{equation}
Clearly, $G_\rho'(\vep) = A_\rho-m\epsilon^{m-1}B_\rho$, so $\max_{\vep>0}  \left[G_\vep(\rho)\right]$ is attained at 
\begin{equation}\label{eqn1.4}
\vep_{max}(\rho) := \left(\frac{A_\rho}{mB_\rho}\right)^{1/(m-1)}.
\end{equation}
In particular if $\rho=\mu_n$, we have
\[\vep_{max}(\mu_n) := \left(\frac{A_{\mu_n}}{mB_{\mu_n}}\right)^{1/(m-1)},\]
where 
\[A_{\mu_n}=n^2m\left(\ib \mu_n^{1/m}\dd y -\left(\frac{\omega_n}{n}\right)^{1/n}\theta_R\right)\]
and
\[B_{\mu_n}=\frac{n^{1+1/n}}{m\omega_n^{1/n}} \left(\ib \mu_n^{1/m}\dd y -\left(\frac{\omega_n}{n}\right)^{1/n}\theta_R\right).\]
Note that we have used above the relation
\begin{equation}\label{new relation}
\ib |y|^m\mu_n \dd x = \left(\frac{n}{\omega_n}\right)^{1/n} \ib \mu_n^{1/m}\dd x - \theta_R.
\end{equation}
Then
\begin{equation}\label{eqn1.5}
\vep_{max}(\mu_n) = (n^{1/m} m\,\omega_n^{1/n})^{1/(m-1)}.
\end{equation}
Choosing $\vep=\vep_{max}(\mu_n)$ in (\ref{2nd-prime}), we have
%\begin{equation}\label{maxim}
\[G_\rho\left(\vep_{max}(\mu_n)\right)  - \ib |\nabla(\log\mu_n))|^n\dd x \leq \ib H_n(u,\mu_n) \dd x + n^nm^{n-1}\omega_n\ib u\dd\mu_n,\]
%\end{equation}
that is, after dividing by $\beta(n)=n^nm^{n-1}\omega_n$,
\begin{eqnarray}\label{maxim}
\lefteqn{\frac{n^2m\,\vep_{max}(\mu_n)}{\beta(n)} \ib \rho^{1/m}\dd y  -\frac{n(\vep_{max}(\mu_n))^m}{m\beta(n)}  \int_{B_R}|y|^m\rho \dd y} \nonumber \\
& & \quad -\frac{1}{\beta(n)} \left[ \ib |\nabla(\log\mu_n))|^n\dd x +  n^{(1+1/m)}m\omega_n^{1/n}\,\vep_{max}(\mu_n)\, \theta_R\right] \nonumber \\
& & \leq \frac{1}{\beta(n)}\ib H_n(u,\mu_n) \dd x + \ib u\dd\mu_n = I_R(u).
\end{eqnarray}
We now simplify the l.h.s of (\ref{maxim}) by using the following basic identities which can be  checked by direct computations:
\[\frac{n^2m\,\vep_{max}(\mu_n)}{\beta(n)} = m(n/\omega_n)^{1/n} =\alpha(n), 
\]
\[\frac{n(\vep_{max}(\mu_n))^m}{m\beta(n)} = 1,\]
\[
 n^{(1+1/m)}m\omega_n^{1/n}\,\vep_{max}(\mu_n) = m^n n^n\omega_n, \]
\[\ib |\nabla(\log\mu_n)|^n\dd x = n^{n-1}m^n\, \omega_n \ib |y|^m\mu_n \dd y,\]
\[\theta_R = \left(\frac{n}{\omega_n}\right)^{1/n} \ib \mu_n^{1/m}\dd y - \ib |y|^m\mu_n \dd y.\]
Then (\ref{maxim}) yields
\[J_R(\rho) - J_R(\mu_n) \leq I_R(u)\]
for all functions $u$ and $\rho$ such that $u\in W^{1,n}_0(B_R)$, $\ib e^u\dd\mu_n = \theta_R$ and $\ib \rho(y) \dd y =\theta_R$.
We conclude the proof as in  Theorem  \ref{theo 2-duality}.
\end{proof}
From Theorem  \ref{theo n-duality}, we obtain the following extension of the 2-dimensional Onofri inequality (\ref{L2 onofri}) to higher dimensions $n\geq 2$.

\begin{thm} ({\bf n-dimensional Euclidean Onofri inequality})\label{theo n-euclidean-onofri}

The following  
holds for all $n\geq 2$:
\begin{equation}\label{Ln onofri}
\frac{1}{\beta (n)}\int_{\R^n} H_n(u,\mu_n) \dd x  +  \int_{\R^n}u\, d\mu_n - \log\left(\int_{\R^n} e^u\,\mbox{d}\mu_n\right)   \geq 0  \quad \forall u\in C^1_c(\R^n),
\end{equation}
where $\dd \mu_n=\mu_n(x)\dd x$ is the probability density on $\R^n$, $\mu_n(x)= \frac{n}{\omega_n (1+|x|^\frac{n}{n-1})^n}$,
$\omega_n$ is the area of the unit sphere in $\R^n$, and $H_n(u,\mu_n)$ is the operator
\[H_n(u,\mu_n):=|\nabla u +\nabla(\log\mu_n)|^n - |\nabla(\log\mu_n))|^n - n |\nabla(\log\mu_n))|^{n-2}\nabla(\log\mu_n) \cdot \nabla u.\]
The optimal constant in this inequality is given by $\beta (n)=n^n (\frac{n}{n-1})^{n-1} \omega_n$.
Moreover $u=0$ is an optimal function.
\end{thm}

\begin{proof}
It is similar to the proof of Theorem \ref{theo 2-euclidean-onofri}. Indeed, by Theorem \ref{theo n-duality}, we have for all $u\in C^1_c(\R^n)$ with support in $B_R$, such that $\ib e^u\dd\mu_n =\theta_R=\frac{R^n}{(1+R^m)^{n/m}}$,
\begin{equation}\label{new-16}
I_R(u) =  \frac{1}{\beta(n)} \int_{B_R} H_n(u,\mu_n)\dd x + \int_{B_R} u(x)\dd \mu_n(x) \geq 0,
\end{equation}
with equality if $u=0$. Letting $R\to\infty$, we have $\theta_R\to 1$; then (\ref{new-16}) gives
\[\frac{1}{\beta(n)} \int_{\R^n} H_n(u,\mu_n)\dd x + \int_{\R^n} u(x)\dd \mu_n(x) \geq 0 = \log\left(\int_{\R^n} e^u\dd\mu_n\right),\]
for all $u\in C^1_c(\R^n)$, with equality if $u=0$. This completes the proof of the theorem.
\end{proof}
From the proof of Theorem  \ref{theo n-duality}, we can also derive the following inequality:

\begin{coro}\label{coro new inequality}
Let $n\geq 2$ be an integer. For $v\in C^1_c(\R^n)$ with compact support in $B_R\subset \R^n$ for some $R>0$, we have
\begin{equation}\label{coro new eqn1}
\left(\frac{n}{\omega_n}\right)^{\frac{n-1}{n}} \,\frac{1}{(1+R^{\frac{n}{n-1}})^n} \ib e^v \dd x + \frac{n-1}{n^2} \ib |\nabla v|^n \dd x \geq  \ib \mu_n^{\frac{n-1}{n}}\dd y,
\end{equation}
where $\mu_n(x)=\frac{n}{\omega_n(1+|x|^{\frac{n}{n-1}})^n}$ and $\omega_n$ is the area of the unit sphere in $\R^n$.

In particular, if $n=2$, then (\ref{coro new eqn1}) gives
\begin{equation}\label{coro new eqn1-2d}
\ib e^v\dd x + \frac{(1+R^2)^2\sqrt{\pi}}{4} \ib |\nabla v|^2\dd x \geq \pi(1+R^2)^2\log(1+R^2). 
\end{equation}
\end{coro}
\begin{proof}
Choosing $\rho=\mu_n$ and $\vep=\vep_{max}(\mu_n)$ in (\ref{2nd-bis}), we have
\begin{equation}\label{coro new eqn2}
G_{\mu_n}(\vep_{max}(\mu_n)) \leq \ib |\nabla (u+\log\mu_n)|^n \dd x \qquad \forall \vep>0 \;\; \mbox{and}  \;\; u
\end{equation}
such that $\ib e^u\dd\mu_n =  \theta_R$ and $u\vert_{\partial B_R} = 0$.
Using the computations in the proof of Theorem  \ref{theo n-duality} and setting $m:=\frac{n}{n-1}$,  we have
\[G_{\mu_n}(\vep_{max}(\mu_n)) = \frac{A_{\mu_n}^n}{n(mB_{\mu_n})^{n-1}} = mn\left(\ib \mu_n^{1/m}\dd y - \left(\frac{\omega_n}{n}\right)^{1/n}\,\theta_R\right).\]
This gives 
\[ mn\left(\ib \mu_n^{1/m}\dd y - \left(\frac{\omega_n}{n}\right)^{1/n}\,\theta_R\right) \leq \ib |\nabla (u+\log\mu_n)|^n \dd x.\]
Set $v:=u+\log\mu_n-\log\left(\mu_n\vert_{\partial B_R}\right)$. We have 
\[\nabla v= \nabla(u+\log\mu_n), \quad v\vert_{\partial B_R}=0, \quad \theta_R=\ib e^u\dd\mu_n = \mu_n\vert_{\partial B_R} \ib e^v \dd x,\] 
where $\mu_n\vert_{\partial B_R} = \frac{n}{\omega_n(1+R^m)^n}$. Then (\ref{coro new eqn2}) reads as
\begin{equation}\label{coro new eqn3}
mn\left(\ib \mu_n^{1/m}\dd y - \left(\frac{\omega_n}{n}\right)^{1/n}\,\frac{n}{\omega_n(1+R^m)^n} \ib e^v\dd x\right)  \leq \ib |\nabla v|^n \dd x.
\end{equation}
This gives (\ref{coro new eqn1}) after simplification. Using $\ib \sqrt{\mu_2} = \sqrt{\pi}\log(1+R^2)$ where  $B_R\subset \R^2$, we get (\ref{coro new eqn1-2d}).
\end{proof}

%%%%%%%%%%%%%%%%%%%%%%%%%%%%%%%%%%%%%%%%%%%%

{\bf Acknowledgments:} Part of this work is done while the three authors were visiting the Fields Institute for Research in  Mathematical Sciences (Toronto), during the Thematic program on variational problems in Physics, Economics and Geometry. The authors would like to thank this institution for its hospitality. The three authors are partially supported by the Natural Sciences and Engineering Research Council of Canada (NSERC).


\begin{thebibliography}{999}

\bibitem{AGK} M. Agueh, N. Ghoussoub, and X. Kang, {\em Geometric inequalities via a general comparison principle for interacting gases}, Geom. Funct. Anal., {\bf 14} (2004), 
215-244.

\bibitem{aubin} T. Aubin, {\em Meilleures constantes dans le th\'eor\`eme d'inclusion de Sobolev et un  th\'eor\`eme de Fredholm non lin\'eaire pour la transformation conforme de la courbure scalaire}, J. Funct. Anal., {\bf 32} (1979), 148-174.

\bibitem{brenier} Y. Brenier, {\em Polar factorization and monotone rearrangement of vector-valued functions}, Comm. Pure Appl. Math. {\bf 44:4} (1991), 375-717.

\bibitem{CNV} D. Cordero-Erausquin, B. Nazaret, and C. Villani, {\em A mass transportation approach to sharp Sobolev and Gagliardo-Nirenberg inequalities}, Adv. Math., {\bf 182} 
(2004), 307-332. 

\bibitem{DD} M. Del-Pino and J. Dolbeault, {\em Best constants for Gagliardo-Nirenberg inequalities and applications to nonlinear diffusions}, J. Math. Pures Appl.  {\bf 90,  81} (2002), 847-875. 

\bibitem {DEJ} J. Dolbeault, M. J. Esteban, and G. Jankowiak, {\em The Moser-Trudinger-Onofri inequality}, Preprint 2014.

\bibitem{GM} N. Ghoussoub, and A. Moradifam, {\em Functional inequalities: new  perspectives and new applications}, Mathematical Surveys and Monographs, Vol. 187 (2013) (American Mathematical Society, Providence, RI).

\bibitem{hong} C. W. Hong, {\em A best constant and the Gaussian curvature}, Proc. Amer. Math. Soc. {\bf 97}, 4, (1986), 737-747. 

\bibitem{mc} R. J. McCann, {\em A convexity principle for interacting gases},  Adv. Math. {\bf 128} (1997) 153-179.

\bibitem{moser} J. Moser, {\em A sharp form of an inequality by N. Trudinger}, Indianna Univ. Math. J. {\bf 20} (1970/71), 1077-1092.

\bibitem{onofri} E. Onofri, {\em On the positivity of the effective action in a theory of random surfaces}, Comm. Math. Phys.  {\bf 86}, (1982), 321-326.

\bibitem{OPS} B. Osgood, R. Phillips, and P. Sarnak, {\em Extremals of determinants of Laplacians}, J. Funct. Anal. Phys. {\bf 80}, 1, (1988), 148-211.

\bibitem{tru} N. Trudinger, {\em On imbeddings into Orlicz spaces and some applications}, Indianna Univ. Math. J. {\bf 17}, (1968), 473-483.


\bibitem{vi} C. Villani, {\em Topics in Optimal Transportation}, Graduate Studies in Math, 58, AMS, Providence, RI, 2003.



\end{thebibliography}
\end{document}